\documentclass[a4paper,12pt]{article}

\usepackage[left=2cm,right=2cm, top=2cm,bottom=3cm,bindingoffset=0cm]{geometry}

\usepackage{verbatim}
\usepackage{amsmath}
\usepackage{amsthm}
\usepackage{amssymb}
\usepackage{delarray}
\usepackage{cite}
\usepackage{hyperref}
\usepackage{mathrsfs}
\usepackage{tikz}
\usetikzlibrary{patterns}
\usepackage{caption}
\DeclareCaptionLabelSeparator{dot}{. }
\newcommand{\al}{\alpha}
\newcommand{\be}{\beta}
\newcommand{\ga}{\gamma}
\newcommand{\de}{\delta}
\newcommand{\la}{\lambda}
\newcommand{\om}{\omega}

\newcommand{\eps}{\varepsilon}

\newcommand{\iy}{\infty}

\theoremstyle{plain}

\numberwithin{equation}{section}

\newtheorem{thm}{Theorem}[section]
\newtheorem{lem}[thm]{Lemma}
\newtheorem{prop}[thm]{Proposition}
\newtheorem{cor}[thm]{Corollary}

\theoremstyle{definition}

\theoremstyle{remark}

\DeclareMathOperator*{\Res}{Res}
\DeclareMathOperator{\rank}{rank}
\DeclareMathOperator{\diag}{diag}

 \allowdisplaybreaks

\begin{document}

\begin{center}
{\large\bf Spectral data asymptotics for the matrix Sturm-Liouville operator}
\\[0.2cm]
{\bf Natalia P. Bondarenko} \\[0.2cm]
\end{center}

\vspace{0.5cm}

{\bf Abstract.} The self-adjoint matrix Sturm-Liouville operator on a finite interval with a boundary condition in the general form is studied. We obtain asymptotic formulas for the eigenvalues and the weight matrices of the considered operator. These spectral characteristics play an important role in the inverse spectral theory. Our technique is based on analysis of analytic functions and on the contour integration in the complex plane of the spectral parameter. In addition, we adapt the obtained asymptotic formulas to the Sturm-Liouville operators on a star-shaped graph with two different types of matching conditions.
  
\medskip


\medskip


\vspace{1cm}

\section{Introduction}

The paper concerns spectral theory of differential operators. In particular, we focus on analysis of matrix Sturm-Liouville operators
in the form $\ell Y = -Y'' + Q(x) Y$, where $Q(x)$ is a matrix function. Operators of such form generalize scalar Sturm-Liouville operators,
which have been studied fairly completely (see the monographs \cite{Mar77, LS91, FY01}).
Matrix Sturm-Liouville operators have applications in various branches of physics.
For example, such operators are used in elasticity theory \cite{BHN95}, they are also applied for description of electromagnetic waves \cite{BS95} and
nuclear structure \cite{Chab04}. Matrix Sturm-Liouville operators also generalize Sturm-Liouville operators on metric graphs
(see \cite{BCFK06, Exner08, PP04, Har05, Now07, Yur16} are references therein). The latter operators are often called quantum graphs. They are applied for modeling
wave propagation through a domain being a thin neighborhood of a graph. Such models are used in organic chemistry, mesoscopic physics, nanotechnology and
other branches of science and engineering.

In this paper, we consider the self-adjoint Sturm-Liouville operator on a finite interval with a boundary condition in the general form defined as follows.
Let $L = L(Q(x), T, H)$ be the boundary value problem for the matrix Sturm-Liouville equation:
\begin{gather} \label{eqv}
-Y''(x) + Q(x) Y(x) = \la Y(x), \quad x \in (0, \pi), \\ \label{bc}
Y(0) = 0, \quad V(Y) := T (Y'(\pi) - H Y(\pi)) - T^{\perp} Y(\pi) = 0,
\end{gather}
where 
\begin{itemize}
\item $Y(x) = [y_j(x)]_{j = 1}^m$ is a vector function; 
\item $Q(x) = [q_{jk}(x)]_{j, k = 1}^m$ is the Hermitian  matrix function, called the {\it potential}, 
such that $Q(x) = Q^{\dagger}(x)$, $q_{jk} \in L_2(0, \pi)$, $j, k = \overline{1, m}$;
\item $\lambda$ is the spectral parameter;
\item $T$ is an orthogonal projector in $\mathbb C^m$, $T \in \mathbb C^{m \times m}$, $T^{\perp} = I_m - T$;
\item $H = H^{\dagger} \in \mathbb C^{m \times m}$, $H = T H T$.
\end{itemize}

Here and below $\mathbb C^m$ and $\mathbb C^{m \times m}$ are the spaces of complex 
$m$-vectors and $m \times m$-matrices, respectively, the symbol $\dagger$ denotes the conjugate transpose:
$A = [a_{jk}]$, $A^{\dagger} = [\overline{a_{kj}}]$,
and $I_m$ is the $m \times m$ unit matrix.

Note that the boundary value problem $L$ is self-adjoint. The boundary condition $V(Y) = 0$ is the self-adjoint 
boundary condition in the most general form. 
For simplicity, we impose the Dirichlet boundary condition at $x = 0$.
The problem with the both boundary conditions in the general form, similar to $V(Y) = 0$, 
can also be investigated by using our methods.

In the major part of the literature on the spectral theory for matrix Sturm-Liouville operators, the Dirichlet boundary conditions $Y(0) = Y(\pi) = 0$ or the Robin boundary conditions $Y'(0) - \Theta_1 Y(0) = Y'(\pi) + \Theta_2 Y(\pi) = 0$, $\Theta_j \in \mathbb C^{m \times m}$, $j = 1, 2$, have been considered, since they are the simplest ones. In particular, in the papers \cite{Pap95, Car99, CK09, Bond11} asymptotic formulas for eigenvalues of Sturm-Liouville operators have been derived. In \cite{CK09, MT10, Bond11, Bond19}, inverse spectral problems have been studied, that consist in reconstruction of matrix Sturm-Liouville operators by using the eigenvalues and the weight matrices. Those problem statements generalize the classical inverse problem of recovering the potential $q(x)$ of the scalar Sturm-Liouville problem 
$$
-y'' + q(x) y = \la y, \quad y(0) = y(\pi) = 0,
$$
from the corresponding eigenvalues $\{ \la_n \}_{n \ge 1}$ and the weight numbers 
$$
\al_n := \int_0^{\pi} y_n^2(x)\, dx, \quad n \ge 1,
$$
where $\{ y_n(x) \}_{n \ge 1}$ are the eigenfunctions, normalized by the condition $y_n'(0) = 1$ (see \cite{Mar77, FY01}). The most remarkable results of \cite{CK09, MT10, Bond11, Bond19} are the characterization theorems for the spectral data of the matrix Sturm-Liouville operator, i.e. necessary and sufficient conditions for the solvability of the corresponding inverse problems. A crucial role in those conditions is played by asymptotic formulas for the eigenvalues and the weight matrices.

Spectral properties of operators with general self-adjoint boundary conditions in the form \eqref{bc} have not been so well-studied yet, since this case is more complex for investigation. Recently Xu \cite{Xu19} proved uniqueness theorems for inverse problems for equation~\eqref{eqv} with the both boundary conditions in the general self-adjoint form. Such operators especially worth to be studied because of their applications to quantum graphs. Inverse problem theory for differential operators on graphs is a rapidly developing field nowadays (see the survey \cite{Yur16}). However, characterization of spectral data is an open problem even for Sturm-Liouville operators on the simplest star-shaped graphs, as well as for the matrix Sturm-Liouville operator with the general self-adjoint boundary conditions.

The goal of the present paper is to make the first step to characterization of the spectral data of the problem \eqref{eqv}-\eqref{bc}. We aim to obtain asymptotic formulas for the eigenvalues and the weight matrices of $L$. In the future, we plan to apply these results to the inverse problem theory.

Let us introduce the spectral data studied in our paper. Further we show that the problem $L$ has a countable set of real eigenvalues. It is convenient to number them as
$\{ \la_{nk} \}_{n \in \mathbb N, \, k = \overline{1, m}}$ in nondecreasing order:
$\la_{n_1, k_1} \le \la_{n_2, k_2}$ for $(n_1, k_1) < (n_2, k_2)$. 
For pairs of integers, we mean by $(n_1, k_1) < (n_2, k_2)$ that $n_1 < n_2$ or $n_1 = n_2$, $k_1 < k_2$.
The double indices $(n, k)$ are used because of the eigenvalue asymptotics (see Theorem~\ref{thm:asymptla}).
Multiple eigenvalues occur in the sequence $\{ \la_{nk} \}_{n \in \mathbb N, \, k = \overline{1, m}}$ multiple times, according to the multiplicities. 

Let $\Phi(x, \la)$ be the $m \times m$ matrix solution of equation~\eqref{eqv}, satisfying the boundary conditions $\Phi(0, \la) = I_m$, $V(\Phi) = 0$. Set $M(\la) := \Phi'(0, \la)$. The matrix functions $\Phi(x, \la)$ and $M(\la)$ are called {\it the Weyl solution} and {\it the Weyl matrix} of the problem $L$, respectively. 
Weyl matrices generalize Weyl functions for scalar differential operators (see \cite{Mar77, FY01}) and play an important role in inverse problem theory of matrix Sturm-Liouville operators (see, e.g., \cite{Bond11, Bond19}).
One can easily show, that the Weyl matrix $M(\la)$ is meromorphic in the $\la$-plane. All its singularities are simple poles,
which coincide with the eigenvalues $\{ \la_{nk} \}_{n \in \mathbb N, \, k = \overline{1, m}}$.
Define {\it the weight matrices}:
\begin{equation} \label{defal}
    \al_{nk} = -\Res_{\la = \la_{nk}} M(\la).
\end{equation}

The collection $\{ \la_{nk}, \al_{nk} \}_{n \in \mathbb N, \, k = \overline{1, m}}$ is called {\it the spectral data} of the problem $L$.

A general approach to derivation of eigenvalue asymptotics was suggested in \cite{Nai68} for arbitrary first-order differential systems.
This approach is based on analysis of analytic characteristic functions in the complex plane of the spectral parameter.
Nevertheless, adaption of general methods to systems of special form (e.g, matrix Sturm-Liouville operators or differential operators on graphs)
usually requires an additional investigation (see, e.g., \cite{Pap95, Car99, CK09, Bond11, MP15}). 
We also mention that matrix Sturm-Liouville operators have been studied on the half-line \cite{AM60, Har05} and on the line \cite{Wad80, Olm85, Bond17-line}.
However, operators on infinite domains usually have a finite number of eigenvalues, therefore problems of eigenvalue asymptotics are not relevant to them.

The main difficulty in derivation of the spectral data asymptotics for the problem~\eqref{eqv}-\eqref{bc} is that its spectrum can contain infinitely many groups of multiple and/or asymptotically multiple eigenvalues (see Theorem~\ref{thm:asymptla} for details). Consequently, it is impossible to obtain separate asymptotic formulas for the sequences $\{ \al_{nk} \}_{n \in \mathbb N}$, $k = \overline{1, m}$. In the present paper, we overcome this difficulty, by developing the ideas of \cite{Bond16, Bond17, Bond19}. The eigenvalues are grouped by asymptotics, and asymptotic formulas are derived for the sums of the weight matrices, corresponding to each eigenvalue group.

The paper is organized as follows. Section~2 is devoted to the asymptotic behavior of the eigenvalues. Its main result is Theorem~\ref{thm:asymptla}, proof of which is based on complex analysis of an analytic characteristic function and on the matrix Rouche's Theorem. In Section~3, we define the sums of the weight matrices, corresponding to eigenvalue groups. We derive asymptotic formulas for those sums (Theorems~\ref{thm:asymptals} and~\ref{thm:asymptal}), by developing methods of Section~3 and by using contour integration. In Section~4, our results for the matrix Sturm-Liouville operator are adapted to the Sturm-Liouville operators on the star-shaped graph with two types of matching conditions: $\de$-coupling and $\de'$-coupling, arising in applications.

\section{Eigenvalue asymptotics}

The goal of this section is to derive asymptotic of eigenvalues of the problem $L$. First we need to introduce some notations.

Denote $p = \rank(T)$, $1 \le p \le m - 1$. Clearly, $\rank(T^{\perp}) = m - p$. Without loss of generality, we assume that
\begin{equation} \label{defT}
    T = \begin{bmatrix}
           I_p & 0 \\ 0 & 0
	\end{bmatrix}, \quad
    T^{\perp} = \begin{bmatrix}
                     0 & 0 \\ 0 & I_{m - p}
               \end{bmatrix}, \quad 	
    H = \begin{bmatrix}
            h & 0 \\ 0 & 0
        \end{bmatrix}.
\end{equation}
One can achieve this condition, applying an appropriate unitary transform $U = (U^{\dagger})^{-1}$:
\begin{equation} \label{transformU}
    \tilde T = U^{\dagger} T U, \quad \tilde T^{\perp} = U^{\dagger} T^{\perp} U, \quad
    \tilde H = U^{\dagger} H U, \quad \tilde Q(x) = U^{\dagger} Q(x) U, \quad \tilde Y(x) = U^{\dagger} Y(x).
\end{equation}

Denote 
$$
\om = \frac{1}{2} \int_0^{\pi} Q(x) \, dx = \begin{bmatrix} \om_{11} & \om_{12} \\ \om_{21} & \om_{22} \end{bmatrix},
$$
where $\om_{11} \in \mathbb C^{p \times p}$, $\om_{22} \in \mathbb C^{(m - p) \times (m - p)}$. Obviously,
these matrices are Hermitian: $\om_{11} = \om_{11}^{\dagger}$, $\om_{22} = \om_{22}^{\dagger}$.

Denote by $S(x, \la)$ and $C(x, \la)$ the matrix solutions of equation~\eqref{eqv}, satisfying the initial conditions
$$
   S(0, \la) = C'(0, \la) = 0, \quad S'(0, \la) = C(0, \la) = I_m.
$$
The matrix functions $S^{(\nu)}(x, \la)$ and $C^{(\nu)}(x, \la)$ are entire in $\la$-plane for each fixed $x \in [0, \pi]$, $\nu = 0, 1$.
The eigenvalues of the boundary value problem $L$ coincide with the zeros of 
the {\it characteristic function} $\Delta(\la) := \det(V(S(x, \la)))$, counting with their multiplicities.
This fact can be proved similarly to \cite[Lemma~5]{Bond19}.
Since the problem $L$ is self-adjoint, its eigenvalues are real.
Obviously, the function $\Delta(\la)$ is entire. 
In order to study its asymptotic behavior, we need the relations for $S^{(\nu)}(\pi, \la)$ and $C^{(\nu)}(\pi, \la)$, $\nu = 0, 1$, provided below.

Denote $\rho := \sqrt{\la}$, $\tau := \mbox{Im}\,\rho$. Using the transformation operators (see \cite{AM60}), one can obtain the following representations:
\begin{equation} \label{asymptCS}
\arraycolsep=1.4pt\def\arraystretch{2.2}
\left.
\begin{array}{l}
S(\pi, \la) = \dfrac{\sin \rho \pi}{\rho}   I_m - \dfrac{\cos \rho \pi}{\rho^2}   \om + \dfrac{1}{2\rho^2} \displaystyle\int_0^{\pi} \cos \rho (\pi - 2t) Q(t) \, dt + \dfrac{K(\rho)}{\rho^3}, \\
S'(\pi, \la) = \cos \rho \pi   I_m + \dfrac{\sin \rho \pi}{\rho}   \om - \dfrac{1}{2 \rho} \displaystyle\int_0^{\pi} \sin \rho (\pi - 2t) Q(t) \, dt + \dfrac{K(\rho)}{\rho^2}, \\
C(\pi, \la) = \cos \rho \pi   I_m + \dfrac{\sin \rho \pi}{\rho}   \om + \dfrac{1}{2 \rho} \displaystyle\int_0^{\pi} \sin \rho(\pi - 2 t) Q(t) \, dt + \dfrac{K(\rho)}{\rho^2}, \\
C'(\pi, \la) = - \rho \sin \rho \pi   I_m + \cos \rho \pi   \om + \dfrac{1}{2} \displaystyle\int_0^{\pi} \cos \rho (\pi - 2 t) Q(t) \, dt + \dfrac{K(\rho)}{\rho}. \\
\end{array}
\right\}
\end{equation} 
Here the same notation $K(\rho)$ is used for various matrix functions, entire by $\rho$ and satisfying the conditions
$$
    K(\rho) = O(\exp(|\tau| \pi)),  \quad \int_{-\iy}^{\iy} \| K(\rho) \|^2 \, d\rho < \iy.
$$
Throughout the paper, we use the matrix norm, induced by the Eucliden vector norm in $\mathbb C^m$, i.e. $\| A \|$ equals to the square root of the largest eigenvalue of the matrix $A^{\dagger} A$.

\begin{thm} \label{thm:asymptla}
The spectrum of the boundary value problem $L$ is a countable set of real eigenvalues
$\{ \la_{nk} \}_{n \in \mathbb N, \, k = \overline{1, m}}$, counted with their multiplicities,
$\la_{n_1, k_1} \le \la_{n_2, k_2}$, if $(n_1, k_1) < (n_2, k_2)$. Moreover, the following asymptotic
relations hold for $\rho_{nk} := \sqrt{\la_{nk}}$, $n \in \mathbb N$:
\begin{align} \label{asymptla1}
& \rho_{nk} = n - \frac{1}{2} + \frac{z_k}{\pi (n - 1/2)} + \frac{\varkappa_{nk}}{n}, \quad k = \overline{1, p}, \\ \label{asymptla2}
& \rho_{nk} = n + \frac{z_k}{\pi n} + \frac{\varkappa_{nk}}{n}, \quad k = \overline{p + 1, m},
\end{align}
where $\{ z_k \}_{k = 1}^p$ and $\{ z_k \}_{k = p + 1}^m$ are the eigenvalues of the Hermitian matrices 
$(\om_{11} - H)$ and $\om_{22}$, respectively, numbered in the nondecreasing order according to their multiplicities:
$z_k \le z_{k + 1}$, $k = \overline{1, p-1}$ and $k = \overline{p + 1, m - 1}$,
and $\{ \varkappa_{nk} \} \in l_2$.
\end{thm}

For the proof of Theorem~\ref{thm:asymptla}, we need several auxiliary results. The following proposition is a matrix
version of Rouche's Theorem (see \cite[Lemma~2.2]{CK09}).

\begin{prop} \label{prop:Rouche}
Let $F(\la)$ and $G(\la)$ be matrix-functions, analytic in the disk $|\la - a| \le r$ and satisfying the condition
$\| G(\la) F^{-1}(\la) \| < 1$ on the boundary $|\la - a| = r$. Then the scalar functions
$\det(F)$ and $\det(F + G)$ have the same number of zeros inside the circle $|\la - a| < r$, counting with multiplicities.
\end{prop}

The next proposition is \cite[Lemma~3]{Bond17}.

\begin{prop} \label{prop:poly}
Let $\{ \de_n \}_{n \ge 1}$ and $\{ \ga_n \}_{n \ge 1}$ be two sequences of nonzero numbers, such that
$$
    \de_n^k + a_{1n} \de_n^{k-1} + a_{2n} \de_n^{k-2} + \dots + a_{k-1, n} \de_n + a_{kn} = 0, \quad n \in \mathbb N,
$$
where
$$
    a_{jn} = \sum_{l = 0}^j O(\ga_n^l) o(\de_n^{j-l}), \quad j = \overline{1, k}.
$$
Then $\de_n = O(\ga_n)$ as $n \to \infty$.
\end{prop}

Now we are ready to obtain rough asymptotics for the eigenvalues.

\begin{lem} \label{lem:rough}
The problem $L$ has a countable set of eigenvalues $\{ \la_{nk} \}_{n \in \mathbb N, \, k = \overline{1, m}}$,
$\lambda_{nk} = \rho_{nk}^2$,
numbered according to their multiplicities and having the following asymptotics:
\begin{equation} \label{rough}
\rho_{nk} = n - \frac{1}{2} + \eps_{nk}, \: k = \overline{1, p}, \qquad
\rho_{nk} = n + \eps_{nk}, \: k = \overline{p + 1, m},
\end{equation}
where $\eps_{nk} = O\left( n^{-1} \right)$ as $n \to \infty$.
\end{lem}

\begin{proof}
Consider the matrix-function $W(\la) := V(S(\pi, \la))$. The eigenvalues of $L$ are the zeros of $\det(W(\la))$.
The asymptotics \eqref{asymptCS} yield
\begin{equation} \label{asymptW}
    W(\la) = T \left( \cos \rho \pi + O\biggl( \frac{\exp(|\tau|\pi)}{\rho} \biggr)\right) -
             T^{\perp} \left( \frac{\sin \rho \pi}{\rho} + O\biggl( \frac{\exp(|\tau|\pi)}{\rho^2}\biggr) \right),
   \quad |\rho| \to \iy.
\end{equation}

Define
$$
    W_0(\la) := T \cos \rho \pi - T^{\perp} \frac{\sin \rho \pi}{\rho}.
$$
Obviously, the function $\det(W_0(\la))$ has the zeros $\{ \la_{nk}^0 \}_{n \in \mathbb N, \, k = 1, 2}$,
$\la_{nk}^0 = \left( \rho_{nk}^0 \right)^2$, 
\begin{equation} \label{rhonk0}
\rho_{nk}^0 = n - \frac{1}{2}, \: k = \overline{1, p}, \qquad \rho_{nk}^0 = n, \, k = \overline{p + 1, m}.
\end{equation}

Consider the region $G_{\de} = \{ \rho \in \mathbb C \colon |\rho - (n - 1/2)| \ge \de, \, |\rho - n| \ge \de, \, n \in \mathbb Z \}$
for a fixed $\de$, $0 < \de < 1/4$. Using the standard estimates (see \cite[Section~1.1]{FY01}):
$$
    |\sin \rho \pi| \ge C \exp(|\tau|\pi), \quad |\cos \rho \pi| \ge C \exp(|\tau| \pi), \quad \rho \in G_{\de},
$$
we derive for $\rho \in G_{\de}$, $\la = \rho^2$, that
$$
    W_0^{-1}(\la) (W(\la) - W_0(\la)) = T (\cos \rho \pi)^{-1} O\left( \frac{\exp(|\tau|\pi)}{\rho}\right) +
    T^{\perp} \frac{\rho}{\sin \rho \pi} O\left( \frac{\exp(|\tau|\pi)}{\rho^2} \right) = 
    O\left( \rho^{-1}\right).
$$

Consequently, for sufficiently large $|\la|$, $\sqrt \la = \rho \in G_{\de}$, we have
$$
\| W_0^{-1}(\la)(W(\la) - W_0(\la)) \| < 1.
$$
Applying the matrix Rouche's theorem (Proposition~\ref{prop:Rouche}), we conclude that the functions
$\det (W(\la))$ and $\det (W_0(\la))$ have the same number of zeros in the disks $\{ \la \colon |\la| < (N-1/4) \}$,
$\{ \rho \colon |\rho - (N - 1/2)| < \de \}$ and $\{ \rho \colon |\rho - N| < \de \}$ for sufficiently large values of
$N$ and arbitrary $\de \in (0, 1/4)$.
Taking~\eqref{rhonk0} into account, we arrive at the assertion of the theorem with $\eps_{nk} = o(1)$ as $n \to \infty$.

It remains to prove that $\eps_{nk} = O\left( n^{-1}\right)$ as $n \to \iy$ for $k = \overline{1, m}$. For definiteness,
consider a fixed $k \in \{ 1, \dots, p \}$. The case $k = \overline{p + 1, m}$ is analogous. Substitute $\la = \la_{nk}$ into the 
asymptotic relation~\eqref{asymptW}:
$$
    W(\la_{nk}) = (-1)^n \left( T \bigl(\sin \eps_{nk} \pi + O\bigl(n^{-1}\bigr)\bigr) + T^{\perp} \biggl( 
     \frac{\cos \eps_{nk} \pi}{n} + O\bigl( n^{-1} \bigr)\biggr)\right), \quad n \in \mathbb N.
$$
Consequently, we obtain
$$
    \det (W(\la_{nk})) = \frac{(-1)^{nm}}{n^{m - p}} P_{nk}(\sin \eps_{nk} \pi) = 0,
$$
where $P_{nk}(x)$ is a polynomial of degree $p$ with the coefficients depending on $n$ and $\eps_{nk}$.
One can easily check that this polynomial satisfies the conditions of Proposition~\ref{prop:poly} with
$\de_n = \eps_{nk}$ and $\ga_n = n^{-1}$. Applying Proposition~\ref{prop:poly},
we conclude that $\eps_{nk} = O\left( n^{-1}\right)$, $n \to \iy$.
\end{proof}

Below the symbol $C$ is used for various positive constants.

\begin{proof}[Proof of Theorem~\ref{thm:asymptla}]
Theorem~\ref{thm:asymptla} follows from Lemma~\ref{lem:rough}, if we derive the asymptotics~\eqref{asymptla1} and~\eqref{asymptla2}
from \eqref{rough}. We focus on the proof of~\eqref{asymptla2}, since it is more complicated.

We need to obtain more precise asymptotics for the zeros of the function $W(\rho^2)$, than~\eqref{asymptla1}.
Instead of the $\rho$-plane, we will work with a new $z$-plane, obtained by the mapping
$$
    \rho_n(z) := n + \frac{z}{\pi n}.
$$
Here $z$ lies in the circle $|z| \le r$ of a sufficiently large radius $r > 0$. 
Below we use the notation $\{ K_n(z) \}_{n \in \mathbb N}$ for various sequences of matrix functions, such that
$$
\sum_{n = 1}^{\iy} \| K_n(z) \|^2 \le C, \quad |z| \le r, 
$$
where the constant $C$ depends on a sequence $\{ K_n(z) \}_{n \in \mathbb N}$, but does not depend on $z$.

One can easily obtain the relations
\begin{equation} \label{cs}
\cos (\rho_n(z) \pi) = (-1)^n \bigl( 1 + O\left( n^{-2} \right) \bigr), \quad
\sin (\rho_n(z) \pi) = \frac{(-1)^n z}{n} \bigl( 1 + O\left( n^{-2} \right) \bigr), \quad n \to \iy, 
\end{equation}
where the estimates $O\left( n^{-2}\right)$ are uniform with respect to $z \colon |z| \le r$.

Using~\eqref{asymptCS} and~\eqref{cs}, we get the asymptotics
\begin{equation} \label{asymptWz}
   W(\rho_n^2(z)) = (-1)^n \left( T \biggl( I_m + \frac{K_n(z)}{n}\biggr) - T^{\perp} \biggl(\frac{z   I_m - \om}{n^2}
   + \frac{K_n(z)}{n^2} \biggr) \right).
\end{equation}
Note that the terms with $T$ and $T^{\perp}$ have different powers of $n$. In order to overcome this difficulty,
we introduce the function
\begin{equation} \label{defRn}
    R_n(z) := (-1)^n (T + n^2 T^{\perp}) W(\rho_n^2(z)).
\end{equation}
Clearly,
\begin{equation} \label{asymptRn}
    R_n(z) = T \left( I_m + \frac{K_n(z)}{n} \right) - T^{\perp} (z   I_m - \om + K_n(z)).
\end{equation}

Consider the main part
\begin{equation} \label{defR}
     R(z) := T - T^{\perp} (z   I_m - \om).
\end{equation}
In the block-matrix form, we have
\begin{equation} \label{blockR}
    R(z) = \begin{bmatrix} I_p & 0 \\ \om_{21} & -(z   I_{m-p} - \om_{22})\end{bmatrix}.
\end{equation}
The determinant $\det(R(z)) = (-1)^{m-p}\det(z   I_{m - p} - \om_{22})$ has the real zeros
$\{ z_k \}_{k = p + 1}^m$, being the eigenvalues of the Hermitian matrix $\om_{22} = \om_{22}^{\dagger}$.
Define the region 
$$
Z_{\de} := \{ z \in \mathbb C \colon |z| \le r, \, |z - z_k| \ge \de, \, k = \overline{p + 1, m} \},
$$ 
where $\de > 0$ is so small that $|z_k - z_l| < \de$ for all $l \ne k$, $l, k = \overline{p + 1, m}$.
The relation~\eqref{blockR} yields
\begin{equation} \label{invR}
R^{-1}(z) = \begin{bmatrix}
                   I_p & 0 \\ (z   I_{m-p} - \om_{22})^{-1} \om_{21} & -(z   I_{m - p} - \om_{22})^{-1} 
                \end{bmatrix}.    
\end{equation}

Consequently, $\| R^{-1}(z) \| \le C$ for $z \in Z_{\de}$. The relations~\eqref{asymptRn} and~\eqref{defR} imply
$R_n(z) - R(z) = K_n(z)$, $n \in \mathbb N$. Hence
$$
    \| R^{-1}(z) \| \cdot \| R_n(z) - R(z) \| < 1
$$
for sufficiently large $n$ and $z \in Z_{\de}$. Applying Proposition~\ref{prop:Rouche}, we conclude that,
for sufficiently large $n$, the function $\det(R(z))$ has exactly $(m - p)$ zeros $\{ z_{nk} \}_{k = \overline{p + 1, m}}$ 
(counted according to their multiplicities),
having the asymptotics 
\begin{equation} \label{znk}
    z_{nk} = z_k + \epsilon_{nk}, \quad \epsilon_{nk} = o(1), \quad n \to \iy, \quad k = \overline{p + 1, m}.
\end{equation}

Introduce the unitary matrices $\mathcal U = (\mathcal U^{\dagger})^{-1} \in \mathbb C^{(m - p) \times (m - p)}$ and 
$\tilde{\mathcal U} \in \mathbb C^{m \times m}$,
such that
\begin{equation} \label{defU}
    \mathcal U \om_{22} \mathcal U^{\dagger} = D := \diag\{ z_k \}_{k = p + 1}^m, \quad
    \tilde {\mathcal U} := \begin{bmatrix} I_p & 0 \\ 0 & \mathcal U \end{bmatrix}.
\end{equation}
Using~\eqref{asymptRn}, we obtain
$$
    \tilde {\mathcal U} R_n(z) \tilde {\mathcal U}^{\dagger} = 
    \begin{bmatrix} I_p & 0 \\ \mathcal U \om_{21} & -(z   I_{m-p} - D) \end{bmatrix} + K_n(z).
$$
Applying the Gaussian elimination, we obtain the following matrix:
\begin{equation} \label{tRn}
     H_n(z) := \begin{bmatrix} I_p & 0 \\ 0 & -(z   I_{m-p} - D) \end{bmatrix} + K_n(z).
\end{equation}

Fix $k \in \{ p + 1, \dots, m \}$.
Let $m_k$ be a multiplicity of the zero $z_k$ of the function $\det (R(z))$:
$m_k = \# \{ s \colon z_s = z_k, \, p + 1 \le s \le m \}$.
Using~\eqref{znk} and~\eqref{tRn}, we derive that $\det(H_n(z_{nk})) = \mathcal P_{nk}(\epsilon_{nk})$,
where $\mathcal P_{nk}$ is a polynomial of degree $m_k$ with coefficients, depending on $n$ and $\epsilon_{nk}$ and satisfying Proposition~\ref{prop:poly}
with $\de_n = \epsilon_{nk}$, $\{ \ga_n \} \in l_2$. Proposition~\ref{prop:poly} yields
$\{ \epsilon_{nk} \} \in l_2$. By construction, the zeros of $\det(H_n(z))$ coincide with the zeros of
$W(\rho_n^2(z))$. Returning to the $\rho$-plane, we obtain~\eqref{asymptla2}.

The proof of the relation~\eqref{asymptla1} is analogous, so we omit the details.
In order to prove~\eqref{asymptla1}, one should consider the mapping
$$
    \tilde \rho_n(z) := n - \frac{1}{2} + \frac{z}{\pi (n - 1/2)}, \quad |z| \le r,
$$
and the matrix function
$$
    \tilde R_n(z) := (-1)^n \left( n - \frac{1}{2}\right) W(\rho_n^2(z))
    = T (z    I_m - \om + K_n(z)) + T^{\perp} \left( I_m + \frac{K_n(z)}{n}\right)
$$
instead of $\rho_n(z)$ and $R_n(z)$, respectively.
\end{proof}

\section{Asymptotics of weight matrices}

The goal of this section is to derive asymptotic formulas for the weight matrices $\{ \al_{nk} \}$, defined by~\eqref{defal}.
The main difficulty in our investigation is the complicated behavior of the spectrum, which can contain infinitely many
multiple and/or asymptotically multiple eigenvalues.
Indeed, if there are equal numbers among $\{ z_k \}_{k = 1}^p$ or $\{ z_k \}_{k = p + 1}^m$ in Theorem~\ref{thm:asymptla},
the corresponding eigenvalue subsequences $\{ \la_{nk} \}_{n \in \mathbb N}$ have the same asymptotics for different values of $k$.
In the latter case, one cannot obtain separate asymptotic formulas for $\al_{nk}$, being the residuals of the Weyl matrix.
In order to manage with this problem, the approach of the papers \cite{Bond16, Bond17, Bond19} is developed. We group eigenvalues by asymptotics
and derive asymptotic formulas for the sums of the weight matrices, corresponding to each group.

Note that, if $\lambda_{n_1, k_1} = \la_{n_2, k_2} = \ldots = \la_{n_l, k_l}$, where $(n_1, k_1) < (n_2, k_2) < \ldots < (n_l, k_l)$ is a group of multiple eigenvalues of the size $l$, then $\al_{n_1, k_1} = \al_{n_2, k_2} = \ldots = \al_{n_l, k_l}$. For every such group, introduce the notations $\al'_{n_1,k_1} = \al_{n_1, k_1}$, $\al'_{n_j, k_j} = 0$, $j = \overline{2, l}$.
Define
\begin{gather*}
     \al_n^{(s)} = \sum_{\substack{k = \overline{1, p} \\ z_k = z_s}} \al'_{nk}, \quad s = \overline{1, p}, \qquad 
     \al_n^{(s)} = \sum_{\substack{k = \overline{p+1, m} \\ z_k = z_s}} \al'_{nk}, \quad s = \overline{p+1, m}, \\
    \al_n^I = \sum_{k = 1}^p \al'_{nk}, \qquad \al_n^{II} = \sum_{k = p + 1}^m \al'_{nk}.
\end{gather*}

Let us derive asymptotic formulas for $\alpha_n^{(s)}$, $s = \overline{1, m}$, $\al_n^I$ and $\al_n^{II}$ as $n \to \iy$. The final results are formulated in Theorems~\ref{thm:asymptals} and~\ref{thm:asymptal}.
We start from the following standard formula for the Weyl solution:
$$
   \Phi(x, \la) = C(x, \la) + S(x, \la) M(\la).
$$
Since $V(\Phi) = 0$, we have
\begin{equation} \label{formM}
   M(\la) = -V^{-1}(S) V(C).
\end{equation}

First, we derive asymptotic formulas for $\al_n^{(s)}$. For definiteness, fix $s \in \{ p + 1, \dots, m\}$, since we have focused on this case in Section~2.
Using \eqref{defal}, \eqref{formM} and the residue theorem, we obtain
\begin{equation} \label{intalns}
    \al_n^{(s)} = -\frac{1}{2 \pi i} \int_{\Gamma_n^{(s)}} M(\la) \, d\la = \frac{1}{2 \pi i} \int_{\Gamma_n^{(s)}} W^{-1}(\la) V(C(x, \la)) \, d\la
\end{equation}
for sufficiently large $n$. Here $\Gamma_n^{(s)}$ is a contour in $\la$-plane, 
enclosing the eigenvalues $\{ \la_{nk} \}_{k = \overline{p+1, m} \colon z_k = z_s}$.

Further we develop the ideas of the proof of Theorem~\ref{thm:asymptla}. We use the change of variables 
$\la = \rho_n^2(z)$, $|z - z_s| = \de$, where $\de > 0$ is sufficiently small.
Clearly,
\begin{equation} \label{dla}
    d\la = \frac{2}{\pi n} \rho_n(z) dz = \frac{2}{\pi} \left(1 + O\bigl( n^{-2}\bigr) \right) dz.
\end{equation}

In view of~\eqref{defRn}, we have
\begin{equation} \label{sm1}
    W^{-1}(\rho_n^2(z)) = (-1)^n R_n^{-1}(z) (T + n^2 T^{\perp}).
\end{equation}
Since $R_n(z) = R(z) + K_n(z)$ and $\| R^{-1}(z)\| \le C$ for $|z - z_s| = \de$, we get
\begin{equation} \label{sm2}
    R_n^{-1}(z) = R^{-1}(z) + K_n(z), \quad |z - z_s| = \de.
\end{equation}

Using~\eqref{asymptCS} and \eqref{cs}, we obtain the asymptotic relation
\begin{equation} \label{asymptVC}
   V(C(x, \rho_n^2(z))) = (-1)^n \left( T(-z   I_m + \om - H + K_n(z)) - T^{\perp} \biggl( I_m + \frac{K_n(z)}{n}\biggr) \right).
\end{equation}

Combining the relations~\eqref{sm1}, \eqref{sm2} and \eqref{asymptVC} all together, we get
\begin{equation} \label{sm3}
    W^{-1}(\rho_n^2(z)) V(C(x, \rho_n^2(z))) = -n^2 R^{-1}(z) (T^{\perp} + K_n(z)).
\end{equation}
Substituting~\eqref{dla} and~\eqref{sm3} into \eqref{intalns}, we obtain the relation
$$
   \al_n^{(s)} = -\frac{2 n^2}{\pi} \left( \frac{1}{2\pi i} \int\limits_{|z - z_s| = \de} R^{-1}(z) T^{\perp} dz + K_n\right).
$$
Here and below the notation $\{ K_n \}$ is used for various matrix sequences, independent of $z$ and such that
$\{ \| K_n(z)\| \} \in l_2$.

Using the relation~\eqref{invR}, we get
$$
-R^{-1}(z) T^{\perp} = \begin{bmatrix} 0 & 0 \\ 0 & (z   I_{m-p} - \om_{22})^{-1}\end{bmatrix}
$$
The residue theorem yields
$$
-\frac{1}{2\pi i} \int\limits_{|z - z_s|=r} R^{-1}(z) T^{\perp}dz = \Res_{z = z_s}
\begin{bmatrix} 0 & 0 \\ 0 & (z   I_{m-p} - \om_{22})^{-1} \end{bmatrix}.
$$

In view of~\eqref{defU}, we have $\om_{22} = \mathcal U^{\dagger} D \mathcal U$. Hence
\begin{equation} \label{res}
\Res_{z = z_s} (z   I_{m-p} - \om_{22})^{-1} = \mathcal U^{\dagger} \Res_{z = z_s} (z   I_{m-p} - D)^{-1} \mathcal U = \mathcal U^{\dagger} T_s \mathcal U, 
\end{equation}
where 
$$
T_s = [T_{s, jk}]_{j, k = 1}^{m-p}, \quad T_{s, jk} = \begin{cases} 1, \quad j = k, \, z_{j+p} = z_s\\ 0, \quad \text{otherwise}. \end{cases}
$$
Although the choice of the matrix $\mathcal U$, satisfying~\eqref{defU}, is not unique, the result of~\eqref{res} does not depend on $\mathcal U$. Introduce the matrix 
\begin{equation} \label{As1}
A^{(s)} = \begin{bmatrix} 0 & 0 \\ 0 & \mathcal U^{\dagger} T_s \mathcal U\end{bmatrix}, \quad s = \overline{p + 1, m}.
\end{equation}
Finally, we arrive at the relation
\begin{equation} \label{alns1}
    \alpha_n^{(s)} = \frac{2n^2}{\pi} (A^{(s)} + K_n), \quad s = \overline{p + 1, m}.
\end{equation}
Similarly, we get
\begin{equation} \label{alns2}
    \alpha_n^{(s)} = \frac{2 (n -1/2)^2}{\pi} (A^{(s)} + K_n), \quad s = \overline{1, p},
\end{equation}
where
\begin{equation} \label{As2}
A^{(s)} = \begin{bmatrix}
\mathcal V^{\dagger} T_s \mathcal V & 0 \\ 0 & 0
\end{bmatrix},
\quad T_s = [T_{s, jk}]_{j, k = 1}^p, \quad
T_{s, jk} = \begin{cases}
                1, \quad j = k, \, z_j = z_s,\\
                0, \quad \text{otherwise}, 
            \end{cases}
\quad s = \overline{1, p},
\end{equation}
and $\mathcal V = (\mathcal V^{-1})^{\dagger} \in \mathbb C^{p \times p}$ is a unitary matrix, such that $\mathcal V (\om_{11} - H) \mathcal V^{\dagger} = E := \diag\{z_k\}_{k = 1}^p$.

Thus, we arrive at the following theorem.

\begin{thm} \label{thm:asymptals}
The matrix sequences $\{ \al_n^{(s)} \}_{n \in \mathbb N}$, $s = \overline{1, m}$, satisfy the asymptotic relations~\eqref{alns1} and~\eqref{alns2}, where the matrices $A^{(s)}$, $s = \overline{1, m}$, are defined by~\eqref{As1} and~\eqref{As2}.
\end{thm}

The definitions~\eqref{As1} and~\eqref{As2} immediately imply the corollary.

\begin{cor}
The matrices $A^{(s)}$ in the asymptotics~\eqref{alns1} and~\eqref{alns2} are Hermitian, non-negative definite: $A^{(s)} = \left( A^{(s)} \right)^{\dagger} \ge 0$, $s = \overline{1, m}$. Furthermore, 
\begin{align*}
 \rank (A^{(s)}) & = \# \{k = \overline{1, p} \colon z_k = z_s \}, \quad s=\overline{1, p}, \\
 \rank (A^{(s)}) & = \# \{ k = \overline{p +1, m} \colon z_k = z_s \}, \quad
 s=\overline{p + 1, m}.
\end{align*}
\end{cor}

The following corollary is important for inverse problem theory.

\begin{cor}
The spectral data $\{ \la_{nk}, \al_{nk} \}_{n \in \mathbb N, \,k = \overline{1, m}}$ uniquely specify the matrices $(\om_{11}-H)$ and $\om_{22}$.
\end{cor}

\begin{proof}
Consider the set of indices
$$
\mathcal S := \{ s =\overline{1, p} \colon s = 1 \: \text{or} \: z_s \ne z_{s-1} \} \cup \{ s= \overline{p + 1, m} \colon s = p+1 \: \text{or} \: z_s \ne z_{s-1}\}. 
$$
Using~\eqref{As1} and~\eqref{As2}, we obtain
\begin{equation} \label{sums}
\sum_{s \in \mathcal S} z_s A^{(s)}  = \begin{bmatrix} 
\mathcal V^{\dagger} E \mathcal V & 0 \\ 0 & \mathcal U^{\dagger} D \mathcal U\end{bmatrix} = \begin{bmatrix} \om_{11} - H & 0 \\ 0 & \om_{22}\end{bmatrix}.
\end{equation}
Thus, one can determine the constants $\{ z_k \}_{k = 1}^m$ and the matrices $\{ A^{(s)} \}_{s = 1}^m$ from the asymptotic formulas \eqref{asymptla1}, \eqref{asymptla2} and \eqref{alns1}, \eqref{alns2}, respectively, and then find $(\om_{11} - H)$ and $\om_{22}$ from~\eqref{sums}.
\end{proof}

Let us proceed to derivation of asymptotic formulas for $\al_n^I$ and $\al_n^{II}$. Theorem~\ref{thm:asymptals} yields
\begin{equation} \label{alrough}
    \al_n^I = \frac{2 (n - 1/2)^2}{\pi} (T + K_n), \quad
    \al_n^{II} = \frac{2n^2}{\pi} (T^{\perp} + K_n).
\end{equation}
However, we can obtain better estimates for the remainder terms in  \eqref{alrough}. For this purpose, we use \eqref{asymptCS} and~\eqref{cs} to get a more precise asymptotics for $W(\rho_n^2(z))$, than \eqref{asymptWz}:
$$
W(\rho_n^2(z)) = (-1)^n \left( T \biggl( I_m + \frac{K_n(z)}{n} - n^{-2} T^{\perp} \biggl( z   I_m - \om + \hat Q_{2n} + \frac{K_n(z)}{n}\biggr)\right),
$$
where
$$
\hat Q_{2 n} := \frac{1}{2} \int_0^{\pi} Q(t) \cos (2 n t) \, dt.
$$
Consequently, the matrix-functions $R_n(z)$, defined by~\eqref{defRn}, can be represented in the form
$$
R_n(z) = R_n^0(z) + \frac{K_n(z)}{n}, \quad R_n^0(z) := T - T^{\perp} (z   I_m - \om + \hat Q_{2n}).
$$

Choose $r$ such that $|z_k| < r$, $k = \overline{p + 1, m}$.
One can easily show that for $|z| = r$ the following relation holds:
\begin{equation} \label{invRn0}
R_n^{-1}(z) T^{\perp} = \left(R_n^0(z)\right)^{-1} T^{\perp} + \frac{K_n(z)}{n} = 
\begin{bmatrix}
 0 & 0 \\ 0 & -(z   I_{m-p} - \om_{22} + \hat Q_{2n})^{-1}
\end{bmatrix} + \frac{K_n(z)}{n}.
\end{equation}

It follows from \eqref{defal}, \eqref{formM} and the residue theorem, that
\begin{equation} \label{intalII}
\al_n^{II} = -\frac{1}{2\pi i} \int_{\Gamma_n} M(\la) \, d\la =
\frac{1}{2\pi i} \int_{\Gamma_n} W(\la) V(C(x, \la)) \, d\la,
\end{equation}
where $\Gamma_n$ is a contour in $\la$-plane, enclosing the eigenvalues $\{ \la_{nk} \}_{k = p + 1}^m$.

Combining \eqref{dla}, \eqref{sm1}, \eqref{asymptVC}, \eqref{invRn0} and \eqref{intalII}, we obtain the relation
\begin{equation} \label{alII}
\al_n^{II} = -\frac{2 n^2}{\pi} \left( \frac{1}{2\pi i} \int\limits_{|z| = r} \bigl( R_n^0(z)\bigr)^{-1} T^{\perp} dz + \frac{K_n}{n} \right).
\end{equation}
Since $\hat Q_{2n} \to 0$ as $n \to \iy$, the matrix-function $\left( R_n^0(z) \right)^{-1}$ for sufficiently large $n$ has the poles $\{ z_{nk}^0 \}_{k = p + 1}^m$, such that $z_{nk}^0 \to z_k$ as $n \to \iy$, $k = \overline{p + 1, m}$. Hence
\begin{equation} \label{limr}
\frac{1}{2 \pi i} \int\limits_{|z| = r} \left( R_n^0(z)\right)^{-1} T^{\perp} dz = \frac{1}{2 \pi i} \lim_{r \to \iy} \int\limits_{|z| = r} \left( R_n^0(z)\right)^{-1} T^{\perp} dz.
\end{equation}
It follows from~\eqref{invRn0}, that
\begin{equation} \label{asymptz}
\left( R_n^0(z)\right)^{-1} T^{\perp} = -z^{-1} \left( I_m + O\bigl( z^{-1}\bigr)\right) T^{\perp}, \quad |z| \to \iy.
\end{equation}
Substituting~\eqref{asymptz} into \eqref{limr}, we get
\begin{equation} \label{resint}
\frac{1}{2 \pi i} \int\limits_{|z| = r} \left( R_n^0(z)\right)^{-1} T^{\perp} dz = -\frac{1}{2\pi i} \int\limits_{|z| = r} \frac{T^{\perp}}{z} \, dz = -T^{\perp}.
\end{equation}
Substituting~\eqref{resint} into~\eqref{alII}, we obtain the asymptotic relation for $\al_n^{II}$. Similarly we derive an asymptotic formula for $\al_n^I$ and, finally, arrive at the following theorem.

\begin{thm} \label{thm:asymptal}
The following asymptotic relations hold
$$
\al_n^I = \frac{2 (n - 1/2)^2}{\pi} \left( T + \frac{K_n}{n}\right), \quad
\al_n^{II} = \frac{2 n^2}{\pi} \left( T^{\perp} + \frac{K_n}{n}\right), \quad
n \in \mathbb N.
$$
\end{thm}

Theorems~\ref{thm:asymptla}, \ref{thm:asymptals} and~\ref{thm:asymptal} yield the following corollary. It is valid for an arbitrary orthogonal projector $T$, not necessarily defined by~\eqref{defT}.

\begin{cor}
Let $\hat L = L(\hat Q(x), T, \hat H)$ be a boundary value problem of the same form~\eqref{eqv}-\eqref{bc} as $L$, but with different coefficients $\hat Q(x)$ and $\hat H$. Suppose that 
$$
T \left( \frac{1}{2} \int_0^{\pi} (Q(x) - \hat Q(x)) \, dx - H + \hat H \right) T = 0, \quad
T^{\perp} \int_0^{\pi} (Q(x) - \hat Q(x)) \, dx \, T^{\perp} = 0.
$$
Then
\begin{gather*}
\rho_{nk} = \hat \rho_{nk} + \frac{\eta_{nk}}{n}, \quad
 \al_n^{(s)} = \hat \al_n^{(s)} + n^2 K_n, \quad \al_n^I = \hat \al_n^I + n K_n, \quad 
\al_n^{II} = \hat \al_n^{II} + n K_n, \\
 \quad n \in \mathbb N, \quad k, s = \overline{1, m}, \quad \{ \eta_{nk} \} \in l_2,
\end{gather*}
where the values $\hat \rho_{nk}$, $\hat \al_n^{(s)}$, $\hat \al_n^I$ and $\hat \al_n^{II}$ are defined by the problem $\hat L$ in the same way as $\rho_{nk}$, $\al_n^{(s)}$, $\al_n^I$ and $\al_n^{II}$ are defined by the problem $L$, respectively, for $n \in \mathbb N$, $k, s = \overline{1, m}$.
\end{cor}

\section{Examples of Sturm-Liouville operators on a graph}

In this section, we apply our results to differential operators on a star-shaped graph. Two types of matching conditions are considered in the internal vertex: $\de$-coupling and $\de'$-coupling (see \cite{Ex97}).

Let $G$ be a star-shaped graph, consisting of $m$ edges $\{ e_j \}_{j = 1}^m$ of length $\pi$. For every edge $e_j$, we introduce a parameter $x_j \in [0, \pi]$. It is supposed that the end $x_j = 0$ corresponds to the boundary vertex and $x_j = \pi$ corresponds to the internal vertex.

Consider the system of Sturm-Liouville equations on the graph $G$:
\begin{equation} \label{StL}
    -y_j''(x_j) + q_j(x_j) y_j(x_j) = \la y_j(x_j), \quad x_j \in (0, \pi), \quad j = \overline{1, m},
\end{equation}
with the Dirichlet boundary conditions in the boundary vertices:
\begin{equation} \label{Dirich}
    y_j(0) = 0, \quad j = \overline{1, m},
\end{equation}
and the $\de$-type matching conditions in the internal vertex:
\begin{equation} \label{dmc}
    y_j(\pi) = y_1(\pi), \quad j = \overline{2, m}, \quad
    \sum_{j = 1}^m y_j'(\pi) = \beta y_1(\pi).
\end{equation}
Here $\{ q_j \}_{j = 1}^m$ are real-valued functions from $L_2(0, \pi)$, $\be \in \mathbb R$. The boundary value problem~\eqref{StL}-\eqref{dmc} can be rewritten in the matrix form, similar to~\eqref{eqv}-\eqref{bc}:
\begin{gather} \label{eqvt}
    -\tilde Y'' + \tilde Q(x) \tilde Y = \la \tilde Y, \quad x \in (0, \pi), \\ \label{bct}
    \tilde Y(0) = 0, \quad \tilde T (\tilde Y'(\pi) - \tilde H \tilde Y(\pi)) - \tilde T^{\perp} \tilde Y(\pi) = 0,
\end{gather}
where 
\begin{gather} \label{tT1}
    \tilde Q(x) = \diag\{q_j(x)\}_{j = 1}^m, \quad \tilde T = [\tilde T_{jk}]_{j, k = 1}^m, \quad \tilde T^{\perp} = [\tilde T^{\perp}_{jk}]_{j, k = 1}^m, \quad \tilde H = \frac{1}{m} \beta \tilde T, \\ \label{tT}
    \tilde T_{jk} = \frac{1}{m}, \quad 
    \tilde T^{\perp}_{jj} = \frac{m-1}{m}, \quad \tilde T^{\perp}_{jk} = -\frac{1}{m}, \quad j, k = \overline{1, m}, \: j \ne k.
\end{gather}

By a unitary transform \eqref{transformU}, the problem~\eqref{eqvt}-\eqref{bct} can be reduced to the form~\eqref{eqv}-\eqref{bc} with the matrices $T$, $T^{\perp}$ and $H$ in the form~\eqref{defT}. Clearly, $p = 1$ and $h = \frac{\beta}{m}$. The matrix $U$ in \eqref{transformU} has to be chosen in such a way, that the columns of $U^{\dagger}$ equal to the orthonormal eigenvectors of the matrices $\tilde T$ and $\tilde T^{\perp}$, corresponding to the eigenvalue $1$. In particular, the first column of $U^{\dagger}$ consists of the equal numbers $\frac{1}{\sqrt{m}}$ and the remaining columns are the orthonormal eigenvectors of $\tilde T^{\perp}$. The choice of $U$ is not unique, but all the possible choices lead to the same asymptotics of the spectral data.

Let us find the matrices $\om_{11}$ and $\om_{22}$. Obviously,
$$
\tilde \om = \frac{1}{2} \int_0^{\pi} \tilde Q(x) \, dx = \diag\{\tilde \om_j\}_{j = 1}^m, \quad
\om = U \tilde \om U^{\dagger}.
$$
The matrix $\om_{11}$ contains only one element. Since the first row of $U$ consists of the numbers $\frac{1}{\sqrt{m}}$, we get
\begin{equation} \label{defz1}
\om_{11} = \frac{1}{m} \sum_{j = 1}^m \tilde \om_j, \quad
z_1 = \om_{11} - \frac{\be}{m}.
\end{equation}
Using~\eqref{transformU}, we derive the relation
\begin{equation} \label{sm5}
T^{\perp} \om T^{\perp} =  T^{\perp} U \tilde \om U^{\dagger} T^{\perp} = U \tilde T^{\perp} \tilde \om \tilde T^{\perp} U^{\dagger}.
\end{equation}
Clearly, the polynomial $\det (z I_m - T^{\perp} \om T^{\perp})$ has the zero $z = 0$, and its other zeros coincide with the eigenvalues $\{ z_k \}_{k = 2}^m$ of the matrix $\om_{22} \in \mathbb C^{(m-1) \times (m - 1)}$, counting with their multiplicities. In view of~\eqref{sm5}, we have
$$
\det (z I_m - T^{\perp} \om T^{\perp}) = \det (z I_m - \tilde T^{\perp} \tilde \om \tilde T^{\perp}).
$$
Using~\eqref{tT}, one can show that 
$$
\det (z I_m - \tilde T^{\perp} \tilde \om \tilde T^{\perp}) = \frac{z}{m} \mathcal P(z), \quad \mathcal P(z) := \frac{d}{dz} \left( \prod_{j = 1}^m (z - \tilde \om_j)\right).
$$

The unitary transform $U$ does not influence the eigenvalues. Thus, we arrive at the following result. 

\begin{thm} \label{thm:graph1}
The spectrum of the boundary value problem \eqref{StL}-\eqref{dmc} is formed by a countable set of real eigenvalues $\{ \tilde \la_{nk} \}_{n \in \mathbb N, \, k = \overline{1, m}}$, such that $\tilde \rho_{nk} = \sqrt{\tilde \lambda_{nk}}$ satisfy the asymptotic formulas 
\begin{align*}
    & \tilde \rho_{n1} = n - \frac{1}{2} + \frac{z_1}{\pi n} + \frac{\tilde \varkappa_{n1}}{n}, \\
    & \tilde \rho_{nk} = n + \frac{z_k}{\pi n} + \frac{\tilde \varkappa_{nk}}{n}, \quad k = \overline{2, m},
\end{align*}
where $z_1$ is defined by \eqref{defz1}, $\{ z_k \}_{k = 2}^m$ are the zeros of the polynomial $\mathcal P(z)$, numbered according to their multiplicities, and $\{ \tilde \varkappa_{nk} \} \in l_2$.
\end{thm}

Define the weight matrices $\{ \tilde \al_{nk} \}_{n \in \mathbb N, \, k = \overline{1, m}}$ of the problem~\eqref{StL}-\eqref{dmc} on the star-shaped graph as the weight matrices of the equivalent matrix problem~\eqref{eqvt}-\eqref{bct}. Let $\{ \tilde \al_n^{(s)} \}_{n \in \mathbb N, \, s = \overline{1, m}}$, and $\{ \tilde \al_n^{II} \}_{n \in \mathbb N}$ be the sums of the weight matrices, defined for the problem~\eqref{eqvt}-\eqref{bct} similarly to the ones in Section~3. Then Theorems~\ref{thm:asymptals} and~\eqref{thm:asymptal} together with the relations~\eqref{transformU} imply the following theorem.

\begin{thm} \label{thm:graph2}
The weight matrices of the problem~\eqref{StL}-\eqref{dmc} satisfy the following asymptotic relations:
\begin{gather*}
\tilde \al_n^{(1)} = \frac{2 (n - 1/2)^2}{\pi} \left( \tilde T + \frac{K_n}{n}\right), \quad
\tilde \al_n^{II} = \frac{2 n^2}{\pi} \left( \tilde T^{\perp} + \frac{K_n}{n}\right), \\
\tilde \al_n^{(s)} = \frac{2 n^2}{\pi} \left( U^{\dagger} A^{(s)} U + K_n \right), \quad s = \overline{2, m}, \quad n \in \mathbb N,
\end{gather*}
where $\{ A^{(s)} \}_{s = 2}^m$ are the matrices, defined by \eqref{As1}.
\end{thm}

The eigenvalue asymptotics of Theorem~\ref{thm:graph1} have been obtained in \cite{MP15} by another method.
In \cite{Kuz18}, asymptotic relations for the main diagonals of the weight matrices $\{ \al_{nk} \}_{n \in \mathbb N, \, k = \overline{1, m}}$ have been derived, which conform to Theorem~\ref{thm:graph2}.

Further we study the case of $\de'$-coupling. Consider the boundary value problem on the star-graph $G$ for the Sturm-Liouville equations~\eqref{StL} with the Dirichlet boundary conditions \eqref{Dirich} and the following $\de'$-type matching conditions:
\begin{equation} \label{dtmc}
    y_j'(\pi) = y_1'(\pi), \quad j = \overline{2, m}, \quad \sum_{j = 1}^m y_j(\pi) = \be y_1'(\pi).
\end{equation}

The conditions~\eqref{dtmc} can be rewritten in the matrix form as follows:
$$
\tilde T^{\perp} \tilde Y'(\pi) - \tilde T (\tilde Y(\pi) - \tilde H \tilde Y'(\pi)) = 0,
$$
where $\tilde T$, $\tilde T^{\perp}$ and $\tilde H$ are defined by~\eqref{tT1}-\eqref{tT}. In the case $\be \ne 0$, we obtain the Robin-type boundary condition
$$
\tilde Y'(\pi) - \Theta \tilde Y(\pi) = 0, \quad \Theta := (\tilde T^{\perp} + \tilde T \tilde H)^{-1} \tilde T.
$$

Matrix Sturm-Liouville operators with such boundary conditions have been studied, e.g., in \cite{Bond19}, so we exclude this case from consideration.

Suppose that $\be = 0$. Then the problem~\eqref{StL}, \eqref{Dirich}, \eqref{dtmc} can be reduced to the matrix form~\eqref{eqv}-\eqref{bc} by the unitary transform with the same matrix $U$, which has been used for the $\de$-coupling case. We obtain $T = U \tilde T^{\perp} U^{\dagger}$, $T^{\perp} = U \tilde T U^{\dagger}$, $p = m - 1$, $H = 0$. Summarizing the arguments above, we arrive at the following theorem, providing the asymptotics for the spectral data $\{ \tilde \la_{nk}, \tilde \al_{nk} \}_{n \in \mathbb N, \, k = \overline{1, m}}$ of the problem \eqref{StL}, \eqref{Dirich}, \eqref{dtmc}.

\begin{thm}
The spectrum of the boundary value problem \eqref{StL}, \eqref{Dirich}, \eqref{dtmc} with $\be = 0$ is formed by a countable set of real eigenvalues $\{ \tilde \lambda_{nk} \}_{n \in \mathbb N, \, k = \overline{1, m}}$, such that $\tilde \rho := \sqrt{\la_{nk}}$ satisfy the asymptotic formulas
\begin{align*}
    & \tilde \rho_{nk} = n - \frac{1}{2} + \frac{z_k}{\pi n} + \frac{\tilde \varkappa_{nk}}{n}, \quad k = \overline{1, m-1}, \\
    & \tilde \rho_{nm} = n + \frac{z_m}{\pi n} + \frac{\tilde \varkappa_{nm}}{n}, 
\end{align*}
where $\{ z_k \}_{k = 1}^{m-1}$ are the zeros of the polynomial $\mathcal P(z)$, $z_m = \frac{1}{m} \sum\limits_{j = 1}^m \tilde \om_j$, $\{ \tilde \varkappa_{nk} \} \in l_2$.

The sums of the weight matrices satisfy the following asymptotic relations:
\begin{gather*}
\tilde \al_n^{(s)} = \frac{2 (n - 1/2)^2}{\pi} (U^{\dagger} A^{(s)} U + K_n), \quad s = \overline{1, m-1}, \quad n \in \mathbb N, \\
\tilde \al_n^I = \frac{2(n-1/2)^2}{\pi} \left(\tilde T^{\perp} + \frac{K_n}{n} \right), \quad
\tilde \al_n^{(m)} = \frac{2 n^2}{\pi} \left( \tilde T + \frac{K_n}{n}\right), \quad n \in \mathbb N,
\end{gather*}
where the matrices $\{ A^{(s)} \}_{s = 1}^{m-1}$ are defined by~\eqref{As2} and $\tilde T$, $\tilde T^{\perp}$ are defined by \eqref{tT}.
\end{thm}

\medskip

{\bf Acknowledgements.} This work was supported by Grant 19-71-00009 of the Russian Science Foundation.

\medskip

\noindent Natalia Pavlovna Bondarenko \\
1. Department of Applied Mathematics and Physics, Samara National Research University, \\
Moskovskoye Shosse 34, Samara 443086, Russia, \\
2. Department of Mechanics and Mathematics, Saratov State University, \\
Astrakhanskaya 83, Saratov 410012, Russia, \\
e-mail: {\it BondarenkoNP@info.sgu.ru}

\end{document}